\documentclass[11pt,english]{amsart}
\usepackage[T1]{fontenc}
\usepackage[latin9]{inputenc}
\usepackage{babel}
\usepackage{amsthm}
\usepackage{amssymb}
\usepackage{setspace}
\usepackage[unicode=true,
 bookmarks=true,bookmarksnumbered=false,bookmarksopen=false,
 breaklinks=false,pdfborder={0 0 1},backref=false,colorlinks=false]
 {hyperref}
\hypersetup{pdftitle={Essential normality and the decomposability of homogeneous submodules},
 pdfauthor={Matthew Kennedy}}

\makeatletter
\numberwithin{equation}{section}
\numberwithin{figure}{section}
\theoremstyle{plain}
\newtheorem{thm}{\protect\theoremname}[section]
  \theoremstyle{plain}
  \newtheorem{cor}[thm]{\protect\corollaryname}
  \theoremstyle{definition}
  \newtheorem{defn}[thm]{\protect\definitionname}
  \theoremstyle{remark}
  \newtheorem{rem}[thm]{\protect\remarkname}
  \theoremstyle{plain}
  \newtheorem{lem}[thm]{\protect\lemmaname}
  \theoremstyle{plain}
  \newtheorem{prop}[thm]{\protect\propositionname}
  \theoremstyle{definition}
  \newtheorem{example}[thm]{\protect\examplename}

\usepackage[T1]{fontenc}
\usepackage{ae,aecompl}

\usepackage{microtype}

\makeatother

  \providecommand{\corollaryname}{Corollary}
  \providecommand{\definitionname}{Definition}
  \providecommand{\examplename}{Example}
  \providecommand{\lemmaname}{Lemma}
  \providecommand{\propositionname}{Proposition}
  \providecommand{\remarkname}{Remark}
\providecommand{\theoremname}{Theorem}

\begin{document}
\global\long\def\Fd{F_{d}^{2}}

\global\long\def\Ad{A_{d}}

\global\long\def\AI{\mathcal{A}_{\mathcal{I}}}

\global\long\def\AId{(\mathcal{A}_{\mathcal{I}})^{**}}

\global\long\def\Ld{F_{d}^{\infty}}

\global\long\def\Ldd{(F_{d}^{\infty})^{**}}

\global\long\def\Ldz{F_{d,0}^{\infty}}

\title[A non-self-adjoint Lebesgue decomposition]{A non-self-adjoint Lebesgue decomposition }

\author{Matthew Kennedy}

\address{School of Mathematics and Statistics, Carleton University, 1125 Colonel
By Drive, Ottawa, Ontario K1S 5B6, Canada}

\email{mkennedy@math.carleton.ca}

\author{Dilian Yang}

\address{Department of Mathematics and Statistics, University of Windsor,
401 Sunset Avenue, Windsor, Ontario N9B 3P4, Canada}

\email{dyang@uwindsor.ca}
\begin{abstract}
We study the structure of bounded linear functionals on a class of
non-self-adjoint operator algebras that includes the multiplier algebra
of every complete Nevanlinna-Pick space, and in particular the multiplier
algebra of the Drury-Arveson space. Our main result is a Lebesgue
decomposition expressing every linear functional as the sum of an
absolutely continuous (i.e. weak-{*} continuous) linear functional,
and a singular linear functional that is far from being absolutely
continuous. This is a non-self-adjoint analogue of Takesaki's decomposition
theorem for linear functionals on von Neumann algebras. We apply our
decomposition theorem to prove that the predual of every algebra in
this class is (strongly) unique. 
\end{abstract}


\thanks{Both authors partially supported by NSERC}

\maketitle

\section{Introduction}

The main result in this paper is a decomposition theorem for bounded
linear functionals on a class of operator algebras that includes the
multiplier algebra of every complete Nevanlinna-Pick space. Results
of this kind can be seen as a noncommutative generalization of the
Yosida-Hewitt decomposition of a measure into completely additive
and purely finitely additive parts, or more classically, the Lebesgue
decomposition of a measure into absolutely continous and singular
parts. 

Takesaki proved in \cite{Tak58} that a bounded linear functional
on a von Neumann algebra can be decomposed uniquely into the sum of
a normal (i.e. weak-{*} continuous) linear functional, and a singular
linear functional that is far from being normal. In \cite{And78},
Ando proved a direct analogue of Takesaki's decomposition theorem
for linear functionals on the algebra $H^{\infty}$, of bounded analytic
functions on the complex unit disk $\mathbb{D}$. More recently, in
\cite{Ued09}, Ueda proved a generalization of Ando's result for finite
maximal subdiagonal algebras, which are ``analytic'' subalgebras of
finite von Neumann algebras introduced by Arveson in \cite{Arv67}
as a noncommutative generalization of the algebra $H^{\infty}$.

A compelling case can be made that the natural function-theoretic
generalization of $H^{\infty}$ is the algebra $H_{d}^{\infty}$ of
multipliers on the Drury-Arveson space $H_{d}^{2}$. The algebra $H_{d}^{\infty}$
is contained in the algebra $H^{\infty}(\mathbb{B}_{d})$ of bounded
analytic functions on the complex unit ball $\mathbb{B}_{d}$ of $\mathbb{C}^{d}$,
but for $d\geq2$ this inclusion is proper, and $H_{d}^{\infty}$
is seemingly much more tractable than $H^{\infty}(\mathbb{B}_{d})$
(see for example \cite{Arv98}) . The Drury-Arveson space $H_{d}^{2}$
and the multiplier algebra $H_{d}^{\infty}$ are universal in the
following sense: Every irreducible complete Nevanlinna-Pick space
embeds into $H_{d}^{2}$, and the corresponding multiplier algebra
arises as the compression of $H_{d}^{\infty}$ onto this embedding
(see \cite{AM00} for details). Examples of complete Nevanlinna-Pick
spaces include the Hardy space and the Dirichlet space on the disk,
the Drury-Arveson space itself, and more generally the class of Besov-Sobolev
spaces on $\mathbb{B}_{d}$. 

One explanation for the tractability of $H_{d}^{\infty}$ is the fact
that $H_{d}^{\infty}$ arises as a quotient of the noncommutative
analytic Toeplitz algebra $F_{d}^{\infty}$ (see for example \cite{DP98b}
and \cite{AP00}). This algebra, introduced by Popescu in \cite{Pop89},
can be viewed as an algebra of noncommutative analytic functions acting
by left multiplication on a Hardy space $\Fd$ of noncommutative analytic
functions. The operator-algebraic structure of $\Ld$, which is now
well understood, turns out to be strikingly similar to that of $H^{\infty}$
(see for example \cite{Pop89,Pop95,AP00} and \cite{DP98a,DP98b,DP99}).

For a weak-{*} closed two-sided ideal $\mathcal{I}$ of $\Ld$, we
let $\AI$ denote the algebra $\AI=\Ld/\mathcal{I}$. These algebras
are the main objects of interest in this paper, for the following
reason: The multiplier algebra of every irreducible complete Nevanlinna-Pick
space arises as the compression of $\Ld$ to a coinvariant subspace,
and this compression is completely isometrically isomorphic and weak-{*}
to weak-{*} homeomorphic to a quotient of $\Ld$ by a two-sided ideal
(see \cite{DP98b,AP00} for details).

Our main result is the following decomposition theorem for linear
functionals on quotients of $\Ld$. A functional is said to be absolutely
continuous if it is weak-{*} continuous, and singular if it is, roughly
speaking, far from being weak-{*} continuous (we give a precise definition
below).

\begin{thm}
[Lebesgue decomposition for quotients of $F_d^\infty$]Let $\mathcal{I}$
be a weak-{*} closed two-sided ideal of $\Ld$, and let $\phi$ be
a bounded linear functional on $\AI$. Then there are unique linear
functionals $\phi_{a}$ and $\phi_{s}$ on $\AI$ such that $\phi=\phi_{a}+\phi_{s}$,
where $\phi_{a}$ is absolutely continuous and $\phi_{s}$ is singular,
and such that 
\[
\|\phi\|\leq\|\phi_{a}\|+\|\phi_{s}\|\leq\sqrt{2}\|\phi\|.
\]
If $d=1$, then the constant $\sqrt{2}$ can be replaced with the
constant $1$. Moreover, these constants are optimal.
\end{thm}

The following result for multiplier algebras of complete Nevanlinna-Pick
spaces is an immediate consequence of Theorem 1.1.
\begin{cor}
[Lebesgue decomposition for multiplier algebras]Let $\mathcal{A}$
be the multiplier algebra of a complete Nevanlinna-Pick space, and
let $\phi$ be a bounded linear functional on $\mathcal{A}$. Then
there are unique linear functionals $\phi_{a}$ and $\phi_{s}$ on
$\mathcal{A}$ such that $\phi=\phi_{a}+\phi_{s}$, where $\phi_{a}$
is absolutely continuous and $\phi_{s}$ is singular, and such that
\[
\|\phi\|\leq\|\phi_{a}\|+\|\phi_{s}\|\leq\sqrt{2}\|\phi\|.
\]

\end{cor}

We first prove that Theorem 1.1 holds for $\Ld$. The proof for quotients
of $\Ld$ requires the following generalization of the classical F.
\& M. Riesz theorem, which is similar in spirit to the noncommutative
F. \& M. Riesz-type theorems proved by Exel in \cite{Exe90} for operator
algebras with the Dirichlet property, and by Blecher and Labuschagne
in \cite{BL07} for maximal subdiagonal algebras.

\begin{thm}
[Extended F. \& M. Riesz Theorem]Let $\phi$ be a bounded linear
functional on $\Ld$, and let $\phi=\phi_{a}+\phi_{s}$ be the Lebesgue
decomposition of $\phi$ into absolutely continuous and singular parts
as in Theorem 1.1. Let $\mathcal{I}$ be a weak-{*} closed two-sided
ideal of $\Ld$. If $\phi$ is zero on $\mathcal{I}$, then $\phi_{a}$
and $\phi_{s}$ are both zero on $\mathcal{I}$.
\end{thm}

Grothendieck proved in \cite{Gro55} that $L^{1}$ is the unique predual
of $L^{\infty}$ (up to isometric isomorphism). Soon after, in \cite{Sak56},
Sakai generalized Grothendieck's result by proving that the predual
of every von Neumann algebra is unique. In fact, this latter result
follows from the proof of Sakai's characterization of von Neumann
algebras as $\mathrm{C}^{*}$-algebras which are dual spaces.

The uniqueness of the predual of a von Neumann algebra can also be
proved using Takesaki's decomposition theorem from \cite{Tak58} (see
for example the proof of Corollary 3.9 of \cite{Tak02}). A similar
idea was used by Ando in \cite{And78}, to prove the uniqueness of
the predual of $H^{\infty}$, and more recently, by Ueda in \cite{Ued09},
to prove that the predual of every maximal subdiagonal algebra is
unique.

Inspired by these results, we apply Theorem 1.3 to prove that the
predual of every quotient $\AI$ is (strongly) unique.

\begin{thm}
Let $\mathcal{I}$ be a weak-{*} closed two-sided ideal of $\Ld$.
Then the algebra $\mathcal{A}_{\mathcal{I}}$ has a strongly unique
predual.
\end{thm}

It follows immedately from Theorem 1.4 that the multiplier algebra
of every complete Nevanlinna-Pick space has a unique predual.

\begin{cor}
The multiplier algebra of every complete Nevanlinna-Pick space has
a strongly unique predual.
\end{cor}

In particular, Corollary 1.5 implies that the multiplier algebra $H_{d}^{\infty}$
on the Drury-Arveson space has a unique predual. We believe this result
is especially interesting in light of the fact that, for $d\geq2$,
the uniqueness of the predual of $H^{\infty}(\mathbb{B}_{d})$ is
an open problem.

In addition to this introduction, this paper has five other sections.
In Section 2, we provide a brief review of the requisite background
material. In Section 3, we prove the Lebesgue decomposition for $\Ld$,
and give an example showing that the constant in the statement of
the theorem is optimal. In Section 4, we prove the extended F. \&
M. Riesz Theorem. In Section 5, we prove the Lebesgue decomposition
theorem for quotients of $\Ld$, and hence for multiplier algebras
of complete Nevanlinna-Pick spaces. In Section 6, we use the Lebesgue
decomposition theorem to prove that the predual of every quotient
of $\Ld$ is unique, and hence that the predual of the multiplier
algebra of every complete Nevanlinna-Pick space is unique.

\section{Preliminaries}

\subsection{The noncommutative analytic Toeplitz algebra}

For fixed $1\leq d\leq\infty$, let $\mathbb{C}\langle Z\rangle=\mathbb{C}\langle Z_{1},\ldots,Z_{d}\rangle$
denote the algebra of noncommutative polynomials in the variables
$Z_{1},\ldots,Z_{d}$. As a vector space, $\mathbb{C}\langle Z\rangle$
is spanned by the set of monomials 
\[
\{Z_{w}=Z_{w_{1}}\cdots Z_{w_{n}}\mid w=w_{1}\cdots w_{n}\in\mathbb{F}_{d}^{*},\ n\geq0\},
\]
where $\mathbb{F}_{d}^{*}$ denotes the free semigroup generated by
$\{1,\ldots,d\}$. The noncommutative Hardy space $F_{d}^{2}$ is
the Hilbert space obtained by completing $\mathbb{C}\langle Z\rangle$
in the natural inner product
\[
\langle Z_{w},Z_{w'}\rangle=\delta_{w,w'},\quad w,w'\in\mathbb{F}_{d}^{*}.
\]
Equivalently, $F_{d}^{2}$ is the Hilbert space consisting of noncommutative
power series with square summable coefficients,
\[
F_{d}^{2}=\left\{ \sum_{w\in\mathbb{F}_{d}^{*}}a_{w}Z_{w}\mid\sum_{w\in\mathbb{F}_{d}^{*}}|a_{w}|^{2}<\infty\right\} .
\]
We think of the elements of $F_{d}^{2}$ as noncommutative analytic
functions.

Every element in $F_{d}^{2}$ gives rise to a multiplication operator
on $F_{d}^{2}$ in the following way (note that in this noncommutative
setting, it is necessary to specify whether multiplication occurs
on the left or the right). For $F$ in $\Fd$, the left multiplication
operator $L_{F}$ is defined by
\[
L_{F}G=FG,\quad G\in H_{d}^{2}.
\]
The operator $L_{F}$ is not necessarily bounded in general, simply
because the product of two elements in $F_{d}^{2}$ is not necessarily
contained in $F_{d}^{2}$. However, it is always densely defined on
$\mathbb{C}\langle Z\rangle$.

The noncommutative analytic Toeplitz algebra $\Ld$ is the noncommutative
multiplier algebra of $\Fd$. It consists precisely of the functions
$F$ in $H_{d}^{2}$ such that the corresponding left multiplication
operator is bounded,
\[
F_{d}^{\infty}=\{F\in H_{d}^{2}\mid FG\in H_{d}^{2},\ \forall G\in H_{d}^{2}\}.
\]
Equivalently, if we identity $F$ in $F_{d}^{\infty}$ with the left
multiplication operator $L_{F}$ on the Hilbert space $F_{d}^{2}$,
then $F_{d}^{\infty}$ is obtained as the closure of $\mathbb{C}\langle Z\rangle$
in the weak-{*} topology on $\mathcal{B}(F_{d}^{2})$. The noncommutative
disk algebra $A_{d}$ is the closure of $\mathbb{C}\langle Z\rangle$
in the norm topology. Note that it is properly contained in $\Ld$.

The algebras $\Ad$ and $\Ld$ were introduced by Popescu in \cite{Pop96}
and \cite{Pop95} respectively. For $d=1$, $F_{d}^{2}$ can be identified
with the classical Hardy space $H^{2}$, $F_{d}^{\infty}$ can be
identified with the classical algebra of bounded analytic functions,
and $A_{d}$ can be identified with the classical disk algebra of
functions that are analytic on $\mathbb{D}$ with continuous extensions
to the boundary.

\subsection{The structure of an isometric tuple}
\begin{defn}
Let $V=\left(V_{1},\ldots,V_{d}\right)$ be an isometric tuple. Then\end{defn}
\begin{enumerate}
\item $V$ is a \emph{unilateral shift} if it is unitarily equivalent to
a multiple of $L_{Z}=(L_{Z_{1}},\ldots,L_{Z_{d}})$,
\item $V$ is \emph{absolutely continuous} if the unital weak operator closed
algebra $\mathrm{W}(V_{1},\ldots,V_{d})$ generated by $V_{1},\ldots,V_{d}$
 is algebraically isomorphic to the noncommutative analytic Toeplitz
algebra $\Ld$,
\item $V$ is \emph{singular} if the weakly closed algebra $\mathrm{W}\left(V_{1},\ldots,V_{d}\right)$
is a von Neumann algebra, and
\item $V$ is of \emph{dilation type} if it has no summand that is absolutely
continuous or singular.
\end{enumerate}

The next result is from \cite{Ken12}.
\begin{thm}
[Lebesgue-von Neumann-Wold Decomposition]\label{thm:leb-vn-wold-decomp}Let
$V=\left(V_{1},\ldots,V_{d}\right)$ be an isometric $d$-tuple. Then
$V$ can be decomposed as
\[
V=V_{u}\oplus V_{a}\oplus V_{s}\oplus V_{d},
\]
where $V_{u}$ is a unilateral $d$-shift, $V_{a}$ is an absolutely
continuous unitary $d$-tuple, $V_{s}$ is a singular unitary $d$-tuple
and $V_{d}$ is a unitary $d$-tuple of dilation type.
\end{thm}

The next result is from \cite{DKP01}.
\begin{thm}
[Structure Theorem for Free Semigroup Algebras]\label{thm:fsa-struct-thm}Let
$V=\left(V_{1},\ldots,V_{d}\right)$ be an isometric $d$-tuple, and
let $\mathcal{V}=\mathrm{W}(V_{1},\ldots,V_{d})$ denote the unital
weak operator closed algebra generated by $V_{1},\ldots,V_{d}$. Then
there is a maximal projection $P$ in $\mathcal{V}$ with the range
of $P$ coinvariant for $\mathcal{V}$ such that
\begin{enumerate}
\item \textup{$\mathcal{V}P=\cap_{k\geq1}(\mathcal{V}_{0})^{k}$, where
$(\mathcal{V}_{0})^{k}$ denotes the ideal $(\mathcal{V}_{0})^{k}=\sum_{|w|=k}V_{w}\mathcal{V},$}
\item if $P^{\perp}\ne0$, then the restriction of $\mathcal{V}$ to the
range of $P^{\perp}$ is an analytic free semigroup algebra,
\item the compression of $\mathcal{V}$ to the range of $P$ is a von Neumann
algebra, and
\item $\mathcal{V}=P^{\perp}\mathcal{V}P^{\perp}+\mathrm{W}^{*}(V)P.$
\end{enumerate}
\end{thm}

Let $V=V_{u}\oplus V_{a}\oplus V_{s}\oplus V_{d}$ be the Lebesgue-von
Neumann-Wold decomposition of an isometric tuple $V$, as in Theorem
\ref{thm:leb-vn-wold-decomp}, where $V_{u}$ is a unilateral $n$-shift,
$V_{a}$ is an absolutely continuous unitary $n$-tuple, $V_{s}$
is a singular unitary $n$-tuple and $V_{d}$ is a unitary $n$-tuple
of dilation type. Suppose that $V$ is defined on a Hilbert space
$H$, and let $H=H_{u}\oplus H_{a}\oplus H_{s}\oplus H_{d}$ denote
the corresponding decomposition of $H$. By Corollary 2.7 of \cite{DKP01},
there is a maximal invariant subspace $K$ for $V_{d}$ such that
the restriction of $V_{d}$ to $K$ is analytic. The projection $P$
in Theorem \ref{thm:fsa-struct-thm} is determined by $P^{\perp}=P_{H_{u}}\oplus P_{H_{a}}\oplus P_{K}$.

\begin{rem}
\label{rem:proj-reducing-dim-1}For $d=1$, an isometry is the direct
sum of a unilateral shift, an absolutely continuous unitary and a
singular unitary. Theorem \ref{thm:fsa-struct-thm} implies that,
in this case, the structure projection $P$ is the projection onto
the singular unitary part. In particular, this implies that $P$ is
reducing. For $d\geq2$, the proof of Theorem 3.3 shows that $P$
is reducing if and only if there is no summand of dilation type.
\end{rem}

\subsection{The universal representation\label{sub:universal-rep}}

We require the universal representation $\pi_{u}:\Ld\to\mathcal{B}(H_{u})$
of $\Ld$. This can be constructed as in 2.4.4 of \cite{BL05}, as
the restriction of the universal representation of $\mathrm{C}_{\mathrm{max}}^{*}(\Ld)$.
By 3.2.12 of \cite{BL05}, we can identify the double dual $(\Ld)^{**}$
of $\Ld$ with the algebra obtained as the weak-{*} closure of $\pi_{u}(\Ld)$.
We will require the operator algebra structure on $\Ldd$ provided
by this identification. By replacing $\pi_{u}$ by $\pi_{u}^{(\infty)}$
if necessary, we can suppose that $\pi_{u}$ has infinite multiplicity,
and hence that the weak operator topology coincides with the weak-{*}
topology on $\Ldd$.

Let $\phi$ be a bounded linear functional on $\Ld$. By the Hahn-Banach
Theorem, we can extend $\phi$ to a functional on $\mathrm{C}_{\mathrm{max}}^{*}(\Ld)$
with the same norm. Hence by the construction of the universal representation
of $\mathrm{C}_{\mathrm{max}}^{*}(\Ld)$, there are vectors $x$ and
$y$ in $H_{u}$ with $\|x\|\|y\|=\|\phi\|$ such that
\[
\phi(A)=\langle\pi_{u}(A)x,y\rangle,\quad\forall A\in\Ld.
\]
If we identify $\Ld$ with its image $\pi_{u}(\Ld)$ in $(\Ld)^{**}$,
then the functional $\phi$ has a unique weak-{*} continuous extension
to a functional on $(\Ld)^{**}$ with the same norm. We will use this
fact repeatedly.

Since $\pi_{u}$ is the restriction of a {*}-homomorphism of $\mathrm{C}_{\mathrm{max}}^{*}(\Ld)$,
and since the $d$-tuple $(L_{Z_{1}},\ldots,L_{Z_{d}})$ is isometric,
it follows that the $d$-tuple $(\pi_{u}(L_{Z_{1}}),\ldots,\pi_{u}(L_{Z_{d}}))$
is also isometric. Since $(\Ld)^{**}$ contains $\pi_{u}(A_{d})$,
it necessarily contains the weak operator closed algebra generated
by $(\pi_{u}(L_{Z_{1}}),\ldots,\pi_{u}(L_{Z_{d}}))$. Let $P_{u}$
denote the projection in $(\Ld)^{**}$ guaranteed by Theorem \ref{thm:fsa-struct-thm}.
We will refer to $P_{u}$ as the universal structure projection in
$\Ldd$.

\begin{rem}
Let $\mathcal{S}$ denote the unital weak operator closed algebra
generated by  $\pi_{u}(L_{Z_{1}}),\ldots,\pi_{u}(L_{Z_{d}})$. From
above we have $\mathcal{S}\subseteq(\Ld)^{**}$, and one might guess
that $\mathcal{S}=(\Ld)^{**}$. However, this is not the case. Indeed,
let $\phi$ be a bounded nonzero functional on $\Ld$ that is zero
on the noncommutative disk algebra $\Ad$. Then as above, there are
vectors $x$ and $y$ in $H_{u}$ such that
\[
\phi(A)=\langle\pi_{u}(A)x,y\rangle,\quad\forall A\in\Ld.
\]
Let $\psi$ denote the weak operator continuous functional on $\mathcal{S}$
defined by
\[
\psi(S)=\langle Sx,y\rangle,\quad\forall S\in\mathcal{S}.
\]
Since $\phi$ is zero on $A_{d}$, $\psi$ must be zero on $\pi_{u}(A_{d}).$
Then, since $\pi_{u}(A_{d})$ is weak operator dense in $\mathcal{S}$,
it follows that $\psi(S)=\langle Sx,y\rangle=0$ for all $S$ in $\mathcal{S}$.
But, by assumption, there is $A$ in $\Ld$ such that $\phi(A)=\langle\pi_{u}(A)x,y\rangle\ne0$.
So we see that $\pi_{u}(A)\notin\mathcal{S}$, and hence that the
inclusion $\mathcal{S}\subseteq(\Ld)^{**}$ is proper.
\end{rem}

\section{The Lebesgue decomposition}

In this section, we introduce the definitions of absolutely continuous
and singular linear functionals on the noncommutative analytic Toeplitz
algebra $\Ld$, and establish the first version of the Lebesgue decomposition.
In \cite{DLP05}, Davidson, Li and Pitts proved a Lebesgue-type decomposition
for functionals on the noncommutative disk algebra $\Ad$. Although
the algebra $\Ld$ is bigger than $\Ad$, the next definitions is
closely related to (and directly inspired by) the corresponding definition
for $\Ad$.
\begin{defn}
\label{def:abs-cont-and-sing}Let $\phi$ be a bounded linear functional
on $\Ld$. Then\end{defn}
\begin{enumerate}
\item $\phi$ is \emph{absolutely continuous} if it is weak-{*} continuous,
and
\item $\phi$ is \emph{singular} if $\|\phi\|=\|\phi^{k}\|$ for every $k\geq1$,
where $\phi^{k}$ denotes the restriction of $\phi$ to the ideal
of $\Ld$ generated by $\{L_{Z_{w}}\mid|w|=k\}.$
\end{enumerate}

Let $\phi$ be a bounded linear functional on $\Ld$. Then as in Section
\ref{sub:universal-rep}, there are vectors $x$ and $y$ in $H_{u}$
with $\|x\|\|y\|=\|\phi\|$ such that
\[
\phi(A)=\langle\pi_{u}(A)x,y\rangle,\quad\forall A\in\Ld.
\]
We will write $\phi P_{u}$ and $\phi P_{u}^{\perp}$ for the linear
functionals defined on $\Ld$ by
\begin{gather*}
(\phi P_{u})(A)=\langle\pi_{u}(A)P_{u}x,y\rangle,\quad\forall A\in\Ld,\\
(\phi P_{u}^{\perp})(A)=\langle\pi_{u}(A)P_{u}^{\perp}x,y\rangle,\quad\forall A\in\Ld,
\end{gather*}
where $P_{u}$ denotes the universal structure projection from Section
\ref{sub:universal-rep}. The purpose of the next result is to verify
that $\phi P_{u}$ and $\phi P_{u}^{\perp}$ are well defined.

\begin{lem}
Let $\phi$ be a bounded linear functional on $\Ld$. Then the functionals
$\phi P_{u}$ and $\phi P_{u}^{\perp}$, as defined above, do not
depend on the choice of vectors $x$ and $y$.\end{lem}
\begin{proof}
Let $x_{1},y_{1}$ and $x_{2},y_{2}$ be pairs of vectors in $H_{u}$
such that
\[
\langle\pi_{u}(A)x_{1},y_{1}\rangle=\langle\pi_{u}(A)x_{2},y_{2}\rangle,\quad\forall A\in\Ld.
\]
Since $\pi_{u}(\Ld)$ is weak-{*} dense in the algebra $(\Ld)^{**}$,
which contains $P_{u}$, it follows immediately that
\[
\langle\pi_{u}(A)P_{u}x_{1},y_{1}\rangle=\langle\pi_{u}(A)P_{u}x_{2},y_{2}\rangle,\quad\forall A\in\Ld,
\]
and similarly that
\[
\langle\pi_{u}(A)P_{u}^{\perp}x_{1},y_{1}\rangle=\langle\pi_{u}(A)P_{u}^{\perp}x_{2},y_{2}\rangle,\quad\forall A\in\Ld.
\]

\end{proof}

\begin{prop}
\label{prop:char-singular}A bounded functional $\phi$ on $\Ld$
is singular if and only if $\phi=\phi P_{u}$.\end{prop}
\begin{proof}
Let $\phi$ be a singular functional on $\Ld$. We can assume that
$\|\phi\|=1$. As in Section \ref{sub:universal-rep}, there are vectors
$x$ and $y$ in $H_{u}$ such that $\|x\|\|y\|=1$ and 
\[
\phi(A)=\langle\pi_{u}(A)x,y\rangle,\quad\forall A\in\Ld.
\]
By the singularity of $\phi$, we can find a sequence $(A_{k})$ of
elements in $\Ld$ such that $\lim\phi(A_{k})\to1$, and such that
each $A_{k}$ belongs to the unit ball of $(F_{d,0}^{\infty})^{k}=\sum_{|w|=k}\Ld L_{Z_{w}}$.
Let $T$ be an accumulation point of the sequence $(\pi_{u}(A_{k}))$
in $\Ldd$, and let $\mathcal{S}$ denote the unital weak operator
closed  algebra generated by $(\pi_{u}(L_{Z_{1}}),\ldots,\pi_{u}(L_{Z_{d}}))$.
It is clear that the weak-{*} closure of the image $\pi_{u}((F_{d,0}^{\infty})^{k})$
of the ideal $(F_{d,0}^{\infty})^{k}$ can be written as $\Ldd\mathcal{S}_{0}^{k}$,
where $\mathcal{S}_{0}$ denotes the ideal in $\mathcal{S}$ generated
by $\pi_{u}(L_{Z_{1}}),\ldots,\pi_{u}(L_{Z_{d}})$. Thus $\pi_{u}(A_{k})$
belongs to $\Ldd\mathcal{S}_{0}^{k}$. By Theorem \ref{thm:fsa-struct-thm},
$\mathcal{S}P_{u}=\cap_{k\geq1}\mathcal{S}_{0}^{k}$. Hence $T$ belongs
to the unit ball of

\[
\bigcap_{k\ge1}\Ldd\mathcal{S}_{0}^{k}=\Ldd\bigcap_{k\ge1}\mathcal{S}_{0}^{k}=\Ldd P_{u}.
\]
In particular, this means that $T=TP_{u}$. Since $\phi(T)=1$, this
gives
\[
\|x\|\|y\|=1=\langle Tx,y\rangle=\langle TP_{u}x,y\rangle\leq\|P_{u}x\|\|y\|\leq\|x\|\|y\|.
\]
Hence $P_{u}x=x$, and it follows that $\phi=\phi P_{u}$.

Conversely, let $\phi$ be a functional on $\Ld$ such that $\phi=\phi P_{u}$.
As before, we can assume that $\|\phi\|=1$, and there are vectors
$x$ and $y$ in $H_{u}$ such that $\|x\|\|y\|=1$ and 
\[
\phi(A)=\langle\pi_{u}(A)x,y\rangle,\quad\forall A\in\Ld.
\]
The fact that $\phi P_{u}=\phi$ implies that we can choose $x$ satisfying
$x=P_{u}x$, and hence that
\[
\phi(A)=\langle\pi_{u}(A)P_{u}x,y\rangle,\quad\forall A\in\Ld.
\]
Let $\psi$ denote the functional on $\Ldd$ defined by
\[
\psi(T)=\langle TP_{u}x,y\rangle,\quad\forall T\in\Ldd,
\]
and for $k\geq1$, let $\psi^{k}$ denote the restriction of $\psi$
to $\Ldd\mathcal{S}_{0}^{k}$. Then as above, 
\[
\Ldd P_{u}=\bigcap_{k\ge1}\Ldd\mathcal{S}_{0}^{k}.
\]
Hence $\|\psi\|=\|\psi^{k}\|$ for every $k\geq1$. It follows that
$\|\phi\|=\|\phi^{k}\|$, where $\phi^{k}$ is defined as in Definition
\ref{def:abs-cont-and-sing}, and hence that $\phi$ is singular.
\end{proof}

\begin{lem}
\label{lem:range-proj}The range of the projection $P_{u}^{\perp}$
is invariant for $\Ldd$.\end{lem}
\begin{proof}
It suffices to show that whenever $x$ and $y$ are vectors in $\Fd$
such that $x=P_{u}^{\perp}x$ and $y=P_{u}y$, and the functional
$\phi$ on $\Ld$ is defined by
\[
\phi(A)=\langle\pi_{u}(A)x,y\rangle,\quad\forall A\in\Ld,
\]
then $\phi=0$. By Theorem \ref{thm:fsa-struct-thm}, the range of
$P_{u}^{\perp}$ is invariant for $\pi_{u}(A_{d})$. Hence $\phi$
is zero on $\Ad$. Let $A$ be an element of $\Ld$. By Corollary
2.6 of \cite{DP98a}, for $k\geq1$, we can write A uniquely as 
\[
A=\sum_{|w|<k}a_{w}L_{Z_{w}}+A',
\]
where $A'$ belongs to $(F_{d,0}^{\infty})^{k}$. The fact that $\phi$
is zero on $\Ad$ implies that $\phi(A)=\phi(A')$. It follows from
Definition \ref{def:abs-cont-and-sing} that $\phi$ is singular.
Hence by Proposition \ref{prop:char-singular}, $\phi=\phi P_{u}$,
i.e.
\[
\phi(A)=\langle\pi_{u}(A)P_{u}x,y\rangle,\quad\forall A\in\Ld.
\]
Since $x=P_{u}^{\perp}x$, it follows that $\phi=0$, as required.
\end{proof}

\begin{prop}
\label{prop:char-abs-cont}Let $\phi$ be a bounded linear functional
on $\Ld$. Then $\phi$ is absolutely continuous if and only if $\phi=\phi P_{u}^{\perp}$.\end{prop}
\begin{proof}
Suppose first that $\phi$ is absolutely continuous. Then it is weak-{*}
continuous, so there are sequences of vectors $(x_{k})$ and $(y_{k})$
in $\Fd$ such that
\[
\phi(A)=\sum\langle Ax_{k},y_{k}\rangle,\quad\forall A\in\Ld.
\]
Since the $d$-tuple $(L_{Z_{1}},\ldots,L_{Z_{d}})$ is equivalent
to a restriction of the unilateral shift part of the $d$-tuple $(\pi_{u}(L_{Z_{1}}),\ldots,\pi_{u}(L_{Z_{d}}))$,
$F_{d}^{2}$ can be identified with a subspace of $H_{u}$, and it
follows that $\phi=\phi P_{u}^{\perp}$.

Conversely, suppose that $\phi=\phi P_{u}^{\perp}$. As in Section
\ref{sub:universal-rep}, there are vectors $x$ and $y$ in $H_{u}$
with $\|x\|\|y\|=\|\phi\|$ such that
\[
\phi(A)=\langle\pi_{u}(A)x,y\rangle,\quad\forall A\in\Ld.
\]
The fact that $\phi=\phi P_{u}^{\perp}$ implies that we can choose
$x$ satisfying $P_{u}^{\perp}x=x$. Since, by Lemma \ref{lem:range-proj},
the range of $P_{u}^{\perp}$ is invariant for $\pi_{u}(\Ld)$, it
follows that for every $A$ in $\Ld$, we have
\[
\phi(A)=\langle\pi_{u}(A)x,y\rangle=\langle P_{u}^{\perp}\pi_{u}(A)P_{u}^{\perp}x,y\rangle=\langle\pi_{u}(A)P_{u}^{\perp}x,P_{u}^{\perp}y\rangle.
\]
Hence we can also choose $y$ satisfying $P_{u}^{\perp}y=y$.

By the construction of $P_{u}$, the restriction of the operators
$\pi_{u}(L_{Z_{1}}),\ldots,\pi_{u}(L_{Z_{d}})$ to the cyclic subspace
generated by $x$ and $y$ is analytic. Thus, by the main result of
 \cite{Ken12}, the weak-{*} closed algebra generated by this restriction
is completely isometrically isomorphic and weak-{*} to weak-{*} homeomorphic
to $\Ld$. It follows that $\phi$ is weak-{*} continuous on $\Ld$.
\end{proof}

\begin{thm}
[Lebesgue Decomposition for $F_d^\infty$]\label{thm:lebesgue-decomp}Let
$\phi$ be a bounded linear functional on $\Ld$. Then there are unique
linear functionals $\phi_{a}$ and $\phi_{s}$ on $\Ld$ such that
$\phi=\phi_{a}+\phi_{s}$, where $\phi_{a}$ is absolutely continuous
and $\phi_{s}$ is singular, and such that 
\[
\|\phi\|\leq\|\phi_{a}\|+\|\phi_{s}\|\leq\sqrt{2}\|\phi\|.
\]
If $d=1$, then the constant $\sqrt{2}$ can be replaced with the
constant $1$.\end{thm}
\begin{proof}
As in Section \ref{sub:universal-rep}, there are vectors $x$ and
$y$ in $H_{u}$ such that $\|x\|\|y\|=\|\phi\|$ and 
\[
\phi(A)=\langle\pi_{u}(A)x,y\rangle,\quad\forall A\in\Ld.
\]
Define $\phi_{a}$ and $\phi_{s}$ by $\phi_{a}=\phi P_{u}^{\perp}$
and $\phi_{s}=\phi P_{u}$ respectively. Then $\phi_{a}$ is absolutely
continuous by Proposition \ref{prop:char-abs-cont}, and $\phi_{s}$
is singular by Proposition \ref{prop:char-singular}. We clearly have
$\phi=\phi_{a}+\phi_{s}$. To see that $\phi_{a}$ and $\phi_{s}$
are unique, suppose that
\[
\phi_{a}+\phi_{s}=\psi_{a}+\psi_{s},
\]
where $\psi_{a}$ is absolutely continuous and $\psi_{s}$ is absolutely
continuous. Then
\[
\phi_{a}-\psi_{a}=\psi_{s}-\phi_{s}.
\]
It is clear that the functional $\phi_{a}-\psi_{a}$ is absolutely
continuous, and Proposition \ref{prop:char-singular} implies that
the functional $\psi_{s}-\phi_{s}$ is singular. Applying Proposition
\ref{prop:char-abs-cont} and Proposition \ref{prop:char-singular}
again, we can therefore write
\[
\phi_{a}-\psi_{a}=(\phi_{a}-\psi_{a})P_{u}^{\perp}=(\psi_{s}-\phi_{s})P_{u}^{\perp}=(\psi_{s}-\phi_{s})P_{u}P_{u}^{\perp}=0.
\]
Hence $\phi_{a}=\psi_{a}$, and it follows similarly that $\phi_{s}=\psi_{s}$.
Finally, we compute
\[
\|\phi\|\leq\|\phi_{a}\|+\|\phi_{s}\|\leq\|Px\|\|y\|+\|P^{\perp}x\|\|y\|\leq\sqrt{2}\|x\|\|y\|=\sqrt{2}\|\phi\|.
\]
If $d=1$, then Remark \ref{rem:proj-reducing-dim-1} implies that
$\Ldd$ is the direct sum of two algebras reduced by $P_{u}$. If
we identify $\Ld$ with its image in $\Ldd$, then the functionals
$\phi$, $\phi_{a}$ and $\phi_{s}$ extend uniquely to weak-{*} continuous
functionals on $\Ldd$ with the same norm. Since $\phi_{a}=\phi_{a}P_{u}^{\perp}$
and $\phi_{s}=\phi_{s}P_{u}$, it follows that in this case, $\|\phi\|=\|\phi_{a}\|+\|\phi_{s}\|$.
\end{proof}

The next example is based on Example 5.10 from \cite{DLP05}. It establishes
that for $d\geq2$, the constant $\sqrt{2}$ in the statement of Theorem
\ref{thm:lebesgue-decomp} is the best possible.
\begin{example}
\label{ex:constant}Define $\phi$ on $\mathbb{C}\langle Z\rangle$
by setting
\[
\phi(L_{Z_{w}})=\begin{cases}
1/\sqrt{2} & \mbox{if}\ w=\varnothing\ \mbox{or}\ w=21^{n}\ \mbox{for}\ n\geq0,\\
0 & \mbox{otherwise},
\end{cases}
\]
and extending by linearity. We will first show that $\phi$ extends
to a bounded linear functional on the noncommutative disk algebra
$A_{2}$. Let $\mathcal{H}_{\phi}$ denote the Hilbert space $\mathbb{C}e\oplus F_{2}^{2}$,
and define a $2$-tuple $S=(S_{1},S_{2})$ on $\mathcal{H}_{\phi}$
by setting
\[
S_{1}=\left(\begin{array}{cc}
I & 0\\
0 & L_{1}
\end{array}\right),\quad S_{2}=\left(\begin{array}{cc}
0 & 0\\
\xi_{\varnothing}e^{*} & L_{2}
\end{array}\right).
\]
It is easy to check that $S$ is isometric. By the universal property
of the noncommutative disk algebra, we obtain a completely isometric
representation $\pi_{\phi}$ of $A_{2}$ satisfying 

\[
\pi_{\phi}(L_{Z_{w}})=S_{w_{1}}\cdots S_{w_{n}},\quad w=w_{1}\cdots w_{n}\in\mathbb{F}_{d}^{*},
\]
and we can extend $\phi$ to $A_{2}$ by
\[
\phi(A)=\langle\pi_{\phi}(A)(e+\xi_{\varnothing})/\sqrt{2},\xi_{\varnothing}\rangle,\quad A\in A_{2}.
\]
From this, it is easy to check that $\|\phi\|\leq1$. 

Let $\mathcal{S}$ denote the unital weakly closed algebra generated
by $S_{1}$ and $S_{2}$. The structure projection from Theorem \ref{thm:fsa-struct-thm}
is the projection $P$ onto $\mathbb{C}e$, which is contained in
$\mathcal{S}$. Hence $\mathcal{S}$ contains the element $B=(S_{2}P+P^{\perp})/\sqrt{2}$
. The results of \cite{Ken11} imply that Theorem 5.4 of \cite{DLP05}
apply to the unital weak operator closed algebra generated by any
isometric tuple. Thus there is a net $(B_{\lambda})$ of elements
in the unit ball of $A_{d}$ such that $\operatorname{w^{*}-lim}\pi_{\phi}(B_{\lambda})=B$
in $\mathcal{S}$. It is easy to check that $\|B\|=1$ and $\langle B(e+\xi_{\varnothing})/\sqrt{2},\xi_{\varnothing}\rangle=1$,
so it follows that $\|\phi\|=1$.

By the Hahn-Banach theorem, we can extend $\phi$ to a functional
on $\Ld$ with the same norm, which we continue to denote by $\phi$.
Let $\phi=\phi_{a}+\phi_{s}$ be the Lebesgue decomposition of $\phi$
into absolutely continuous and singular parts as in Theorem \ref{thm:lebesgue-decomp}.
Then restricted to $A_{d}$, we can write
\begin{gather*}
\phi_{a}(A)=(\phi P^{\perp})(A)=\langle\pi(A)\xi_{\varnothing}/\sqrt{2},\xi_{\varnothing}\rangle,\quad A\in A_{2},\\
\phi_{s}(A)=(\phi P)(A)=\langle\pi(A)e/\sqrt{2},\xi_{\varnothing}\rangle,\quad A\in A_{2}.
\end{gather*}
Letting $B$ be as above, an easy computation gives 
\[
\langle B\xi_{\varnothing}/\sqrt{2},\xi_{\varnothing}\rangle=\langle Be/\sqrt{2},\xi_{\varnothing}\rangle=1/\sqrt{2}.
\]
Arguing as before, this implies $\|\phi_{a}\|\geq1/\sqrt{2}$ and
$\|\phi_{s}\|\geq1/\sqrt{2}$. By Theorem \ref{thm:lebesgue-decomp},
it follows that $\|\phi_{a}\|+\|\phi_{s}\|=\sqrt{2}\|\phi\|.$
\end{example}

\section{The extended F. \& M. Riesz theorem}

The results in this section can be viewed as noncommutative generalizations
of the classical results referred to as the F. \& M. Riesz theorem.
As mentioned in the introduction, results of this kind have been established
in different settings by Exel in \cite{Exe90}, and by Blecher and
Labuschagne in \cite{BL07}. In fact, Blecher and Labuschagne seem
to have anticipated that an F. \& M. Riesz-type theorem should hold
for $\Ld$ (see the introduction of \cite{BL07}).
\begin{thm}
[Extended F. \& M. Riesz Theorem]\label{thm:ann-ideals}Let $\phi$
be a bounded linear functional on $\Ld$, and let $\phi=\phi_{a}+\phi_{s}$
be the Lebesgue decomposition of $\phi$ into absolutely continuous
and singular parts as in Theorem \ref{thm:lebesgue-decomp}. Let $\mathcal{I}$
be a two-sided ideal of $\Ld$. If $\phi$ is zero on $\mathcal{I}$,
then $\phi_{a}$ and $\phi_{s}$ are both zero on $\mathcal{I}$.\end{thm}
\begin{proof}
As in Section \ref{sub:universal-rep}, there are vectors $x$ and
$y$ in $H_{u}$ such that
\[
\phi(A)=\langle\pi_{u}(A)x,y\rangle,\quad\forall A\in\Ld.
\]
By Proposition \ref{prop:char-abs-cont} we can write $\phi_{a}=\phi P_{u}^{\perp}$,
and by Proposition \ref{prop:char-singular} we can write $\phi_{s}=\phi P_{u}$.
If we identify $\Ld$ with its image $\pi_{u}(\Ld)$ in $\Ldd$, then
the functionals $\phi$, $\phi_{a}$ and $\phi_{s}$ each have unique
weak-{*} continuous extensions to functionals on $\Ldd$ with the
same norm.

Let $\mathcal{J}$ denote the ideal in $\Ldd$ obtained by taking
the weak-{*} closure of $\pi_{u}(\mathcal{I})$. Since $\phi$ is
zero on $\mathcal{I}$, it is zero on $\mathcal{J}$. For $A$ in
$\mathcal{I}$, $\pi_{u}(A)P_{u}^{\perp}$ belongs to $\mathcal{J}$,
which implies
\[
0=(\phi P_{u}^{\perp})(A)=\phi_{a}(A).
\]
Hence $\phi_{a}$ is zero on $\mathcal{I}$, and it follows immediately
that $\phi_{s}$ is also zero on $\mathcal{I}$.
\end{proof}

\begin{cor}
[F. \& M. Riesz Theorem]\label{cor:fmriesz}Let $\phi$ be a bounded
linear functional on $\Ld$. If $\phi$ is zero on $F_{d,0}^{\infty}$,
where $F_{d,0}^{\infty}$ denotes the ideal of $\Ld$ generated by
$L_{Z_{1}},\ldots,L_{Z_{d}}$, then $\phi$ is absolutely continuous.\end{cor}
\begin{proof}
Let $\phi=\phi_{a}+\phi_{s}$ be the Lebesgue decomposition of $\phi$
into absolutely continuous and singular parts as in Theorem \ref{thm:lebesgue-decomp}.
By Theorem \ref{thm:ann-ideals}, $\phi_{a}$ and $\phi_{s}$ are
both zero on $F_{d,0}^{\infty}$. By Definition \ref{def:abs-cont-and-sing},
if $\phi_{s}$ is zero on $F_{d,0}^{\infty}$, it is necessarily zero
on all of $\Ld$. Hence $\phi=\phi_{a}$ and $\phi$ is absolutely
continuous.
\end{proof}

\section{Quotient algebras}

For a weak-{*} closed two-sided ideal $\mathcal{I}$ of $\Ld$, let
$\mathcal{A}_{I}$ denote the quotient algebra $\Ld/\mathcal{I}$.
\begin{defn}
Let $\mathcal{I}$ be a weak-{*} closed two-sided ideal of $\Ld$,
and let $\phi$ be a bounded functional on $\mathcal{A}_{\mathcal{I}}$.
Then\end{defn}
\begin{enumerate}
\item $\phi$ is \emph{absolutely continuous} if it is weak-{*} continuous,
and
\item $\phi$ is \emph{singular} if $\|\phi\|=\|\phi^{k}\|$ for every $k\geq1$,
where $\phi^{k}$ denotes the restriction of $\phi$ to the ideal
of $\AI$ generated by $\{\overline{L_{Z_{w}}}\mid|w|=k\}$, where
for a word $w$ in $\mathbb{F}_{d}^{*}$, $\overline{L_{Z_{w}}}$
denotes the image in $\mathcal{A}_{\mathcal{I}}$ of $L_{Z_{w}}$.\end{enumerate}
\begin{thm}
[Lebesgue decomposition for quotients of $F_d^\infty$]\label{thm:lebesgue-decomp-quotients}Let
$\mathcal{I}$ be a weak-{*} closed two-sided ideal of $\Ld$, and
let $\phi$ be a bounded linear functional on $\mathcal{A}_{\mathcal{I}}$.
Then there are unique linear functionals $\phi_{a}$ and $\phi_{s}$
on $\mathcal{A}_{\mathcal{I}}$ such that $\phi=\phi_{a}+\phi_{s}$,
where $\phi_{a}$ is absolutely continuous and $\phi_{s}$ is singular,
and such that 
\[
\|\phi\|\leq\|\phi_{a}\|+\|\phi_{s}\|\leq\sqrt{2}\|\phi\|.
\]
If $d=1$, then the constant $\sqrt{2}$ can be replaced with the
constant $1$.\end{thm}
\begin{proof}
By basic functional analysis, we can lift the functional $\phi$ to
a functional $\psi$ on $\Ld$ with the same norm. Let $\psi=\psi_{a}+\psi_{s}$
be the Lebesgue decomposition of $\psi$ into absolutely continuous
and singular parts as in Theorem \ref{thm:lebesgue-decomp}. The functional
$\psi$ annihilates $\mathcal{I}$, so by Theorem \ref{thm:ann-ideals},
both $\psi_{a}$ and $\psi_{s}$ annihilate $\mathcal{I}$. Hence
$\psi_{a}$ and $\psi_{s}$ induce functionals $\phi_{a}$ and $\phi_{s}$
on $\mathcal{A}_{\mathcal{I}}$ respectively, with the same norm.
Clearly $\phi=\phi_{a}+\phi_{s}$, and the inequality 
\[
\|\phi\|\leq\|\phi_{a}\|+\|\phi_{s}\|\leq\sqrt{2}\|\phi\|
\]
follows from the corresponding inequality in Theorem \ref{thm:lebesgue-decomp}.
The functional $\phi_{a}$ is absolutely continuous since $\psi_{a}$
is absolutely continuous on $\Ld$. To see that $\phi_{s}$ is singular,
simply note that for every $k\geq1$, the ideal $(\mathcal{A}_{\mathcal{I},0})^{k}$
is the image in $\mathcal{A}_{\mathcal{I}}$ of the ideal $(F_{d,0}^{\infty})^{k}$.
\end{proof}

\begin{cor}
[Lebesgue decomposition  for multiplier algebras]Let $\mathcal{A}$
be the multiplier algebra of a complete Nevanlinna-Pick space, and
let $\phi$ be a bounded linear functional on $\mathcal{A}$. Then
there are unique linear functionals $\phi_{a}$ and $\phi_{s}$ on
$\mathcal{A}$ such that $\phi=\phi_{a}+\phi_{s}$, where $\phi_{a}$
is absolutely continuous and $\phi_{s}$ is singular, and such that
\[
\|\phi\|\leq\|\phi_{a}\|+\|\phi_{s}\|\leq\sqrt{2}\|\phi\|.
\]

\end{cor}

\section{Uniqueness of the predual}

Let $X$ and $Y$ be Banach spaces such that $X^{*}=Y$. Then $X$
is said to be a predual for $Y$. Every predual $X$ of $Y$ naturally
embeds into the dual space $Y^{*}$, and a subspace $X$ of $Y^{*}$
is a predual of $Y$ if and only if it satisfies the following properties: 
\begin{enumerate}
\item The subspace $X$ norms $Y$, i.e. $\sup\{|x(y)|\mid x\in X,\ \|x\|\leq1\}=\|y\|$
for all $y$ in $Y$, and
\item The closed unit ball of $Y$ is compact in the $\sigma(Y,X)$ topology.
\end{enumerate}
The space $Y$ is said to have a strongly unique predual if there
is a unique subspace $X$ of $Y^{*}$ such that $Y=X^{*}$. For a
survey on uniqueness results for preduals, we refer the reader to
Godefroy's article \cite{God89}.

In the operator-theoretic setting, the results of Sakai \cite{Sak56},
Ando \cite{And78} and Ueda \cite{Ued09} mentioned in the introduction
established that von Neumann algebras and maximal subdiagonal algebras
have unique preduals. Ruan proved in \cite{Rua92} that an operator
algebra with a weak-{*} dense subalgebra of compact operators has
a unique predual, which applies to, for example, nest algebras and
atomic CSL algebras. Effros, Ozawa and Ruan proved in \cite{EOR01}
that a $\mathrm{W}^{*}$TRO (i.e. a corner of a von Neumann algebras)
has a unique predual. More recently, in \cite{DW11}, Davidson and
Wright proved that a free semigroup algebra has a unique predual.
Note that Davidson and Wright's result applies to $\Ld$, but not
to quotients of $\Ld$.

The following definition was introduced by Godefroy and Talagrand
in \cite{GT80}. Recall that a (formal) series $\sum_{n}y_{n}$ in
a Banach space $Y$ is weakly unconditionally\emph{ }Cauchy if $\sum_{n}\phi(y_{n})<\infty$
for every $\phi\in Y^{*}$.
\begin{defn}
A Banach space $X$ has property (X) if, for every $\phi\in X^{**}\backslash X$,
there is a weakly unconditionally Cauchy sequence $(x_{n})$ in $X^{*}$
such that
\[
\phi\left(\underset{n}{\operatorname{w^{*}-lim}}\sum_{k=1}^{n}x_{k}\right)\ne\sum_{k=1}^{\infty}\phi(x_{k}).
\]

\end{defn}

One reason for the interest in property (X) is the following result
of Godefroy and Talagrand from \cite{GT80}.
\begin{thm}
[Godefroy-Talagrand]A Banach space $X$ with property (X) is the
unique predual of its dual.
\end{thm}

The following definition is closely related to the notion of an M-ideal
in a Banach space (see \cite{HWW93} for more information).

\begin{defn}
A Banach space $X$ is L-embedded if there is a projection $P$ on
the bidual $X^{**}$ with range $X$ such that
\[
\|x\|=\|Px\|+\|x-Px\|,\quad\forall x\in X^{**}.
\]

\end{defn}

The following result of Pfitzner from \cite{Pfi07} implies that every
separable L-embedded space has property (X), and hence that it is
the unique predual of its dual.
\begin{thm}
[Pfitzner]Separable L-embedded spaces have property (X).
\end{thm}

The results of Sakai, Ando and Ueda on decompositions of linear functionals
imply that the preduals of von Neumann algebras and maximal subdiagonal
algebras are L-embedded, and hence by Pfitzner's result from \cite{Pfi07},
that they are unique. However, Example \ref{ex:constant} shows that
quotients of $\Ld$ are not, in general, L-embedded, so we are unable
to use Pfitzner's result. Instead, we give a direct proof that quotients
of $\Ld$ have (strongly) unique preduals.

\begin{thm}
Let $\mathcal{I}$ be a weak-{*} closed two-sided ideal of $\Ld$.
Then the algebra $\mathcal{A}_{\mathcal{I}}$ has a strongly unique
predual.\end{thm}
\begin{proof}
Suppose $E$ is a predual for $\mathcal{A}_{\mathcal{I}}$, identified
with a subspace of $(\mathcal{A}_{\mathcal{I}})^{*}$. By Theorem
\ref{thm:lebesgue-decomp-quotients},
\[
(\mathcal{A}_{\mathcal{I}})^{*}=(\mathcal{A}_{\mathcal{I}})_{a}^{*}\oplus(\mathcal{A}_{\mathcal{I}})_{s}^{*},
\]
where $(\mathcal{A}_{\mathcal{I}})_{a}^{*}$ and $(\mathcal{A}_{\mathcal{I}})_{s}^{*}$
denote the set of absolutely continuous and singular functionals on
$\mathcal{A}_{\mathcal{I}}$ respectively. We want to prove that $E=(\mathcal{A}_{\mathcal{I}})_{a}^{*}$.

Let $\phi$ be a  functional in $E$, and let $\phi=\phi_{a}+\phi_{s}$
be the Lebesgue decomposition of $\phi$ as in Theorem \ref{thm:lebesgue-decomp-quotients}.
We will prove that $\phi_{s}=0$. Suppose to the contrary that $\phi_{s}\ne0$.
By basic functional analysis, we can lift the functional $\phi$ to
a functional $\psi$ on $\Ld$ that is zero on $\mathcal{I}$. Let
$\psi=\psi_{a}+\psi_{s}$ be the Lebesgue decomposition of $\psi$
as in Theorem \ref{thm:lebesgue-decomp}. By Theorem \ref{thm:ann-ideals},
$\psi_{a}$ and $\psi_{s}$ are both zero on $\mathcal{I}$, and by
construction they induce the functionals $\phi_{a}$ and $\phi_{s}$
respectively on the quotient $\AI$. 

It follows from the results of \cite{Ken11} that Theorem 5.4 of \cite{DLP05}
applies to the unital weak operator closed algebra generated by any
isometric tuple. Thus there is a net $(B_{\lambda})$ of elements
in the unit ball of $\Ld$ such that $\operatorname{w^{*}-lim}\pi_{u}(B_{\lambda})=P_{u}$
in $\Ldd$. Since the net $(B_{\lambda}$) is weak-{*} convergent
in $\Ldd$, it is weakly Cauchy in $\Ld$. Since the closed unit ball
of $\Ld$ is compact in the weak-{*} topology, and in particular is
complete, this implies that there is $B$ in the closed unit ball
of $\Ld$ such that $\operatorname{w^{*}-lim}B_{\lambda}=B$ in $\Ld$.
For every weak-{*} continuous functional $\tau$ on $\Ld$, Proposition
\ref{prop:char-abs-cont} implies that 
\[
\tau(B)=\lim_{\lambda}\tau(B_{\lambda})=(\tau P_{u})(1)=0.
\]
Hence $B=0$.

Let $A$ be an element in the unit ball of $\Ld$ such that $\psi_{s}(A)\ne0$.
Since the net $(B_{\lambda})$ is weakly Cauchy in $\Ld$, the image
$(\overline{B_{\lambda}})$ is weakly Cauchy in $\AI$. It follows
that the net $(\overline{AB_{\lambda}})$ is also weakly Cauchy in
$\AI$. Since $E$ is a predual of $\AI$, the closed unit ball of
$\AI$ is compact in the $\sigma(\AI,E)$ topology, and in particular
is complete. Thus, the net $(\overline{AB_{\lambda}})$ converges
in the $\sigma(\AI,E)$ topology to an element $C$ in the unit ball
of $\AI$. By Proposition \ref{prop:char-singular}, we have 
\[
\phi(C)=\lim_{\lambda}\phi(\overline{AB_{\lambda}})=\lim_{\lambda}\psi(AB_{\lambda})=(\psi P_{u})(A)=\psi_{s}(A)\ne0,
\]
so that $C\ne0$. But since $\operatorname{w^{*}-lim}B_{\lambda}=0$
in $\Ld$, it follows that $\operatorname{w^{*}-lim}\overline{AB_{\lambda}}=0$
in $\AI$. So for every $\tau$ in $(\AI)_{a}^{*}$, we necessarily
have
\[
\tau(C)=\lim_{\lambda}\tau(\overline{AB_{\lambda}})=0.
\]
Since $(\AI)_{a}^{*}$ separates points, this implies that $C=0$,
which gives a contradiction. Thus $\phi=\phi_{a}$, meaning $\phi$
is absolutely continuous. 

Since $\phi$ was arbitrary, it follows from above that every functional
in $E$ is absolutely continuous, i.e. that $E$ is contained in $(\AI)_{a}^{*}$.
If it were the case that $E\ne(\AI)_{a}^{*}$, then we could apply
the Hahn-Banach theorem to separate $E$ from $(\AI)_{a}^{*}$ with
an element of $\AI$. But the fact that $E$ is a predual of $\AI$
means in particular it must norm $\AI$, so this is impossible. Therefore,
we conclude that $E=(\AI)_{a}^{*}$, and hence that $(\AI)_{a}^{*}$
is the unique predual of $\AI$. 
\end{proof}

\begin{cor}
The multiplier algebra of every complete Nevanlinna-Pick space has
a strongly unique predual.
\end{cor}

\section*{Acknowledgements}

The authors are grateful to Ken Davidson and Adam Fuller for their
helpful comments and suggestions.

\end{document}